\documentclass[11pt,a4paper]{amsart}
\usepackage{
amssymb,
amsmath}

\usepackage[bookmarks,colorlinks,pagebackref]{hyperref}

\numberwithin{equation}{section}
\newtheorem{theorem}{Theorem}[section]
\newtheorem{lemma}[theorem]{Lemma}

\newtheorem{corollary}[theorem]{Corollary}

\theoremstyle{definition}
\newtheorem{definition}[theorem]{Definition}
\newtheorem{example}[theorem]{Example}

\newtheorem{remark}[theorem]{Remark}

\newtheorem{notation}[theorem]{Notation}
\theoremstyle{example}

\newcommand\Supp{\operatorname{Supp}}

\newcommand\Rad{\operatorname{Rad}}
\newcommand\Ker{\operatorname{Ker}}
\newcommand\Coker{\operatorname{Coker}}

\newcommand{\au}{\underline a}
\newcommand{\q}{\mathfrak{q}}
\newcommand\grade{\operatorname{grade}}

\newcommand\depth{\operatorname{depth}}

\newcommand{\Cech}{{\Check {C}}_{\au}}

\author[Khadam]{M. Azeem Khadam}
\address{Abdus Salam School of Mathematical Sciences, GC University Lahore, 68-B, New Muslim Town, Lahore 54600, Pakistan}
\email{azeem.khadam@sms.edu.pk}

\thanks{The author is supported by ASSMS, GC University Lahore under a postdoctoral fellowship.}

\title[Cohen Macaulayness of the form Module]{A criterion
	of Cohen Macaulayness of the form module}

\begin{document}

\begin{abstract} 
Let $\mathfrak{q}$ be an ideal of a Noetherian local ring $(A,\mathfrak{m})$ and $M$ a non-zero finitely generated $A$-module. We present a criterion of Cohen-Macaulayness of the form module $G_M(\q)$ in terms of (non-)vanishing of a variation of local cohomology introduced in \cite{KSch}.
\end{abstract}


\maketitle

\section{Introduction} 
Let $(A,\mathfrak{m})$ be a Noetherian local ring and $M$ a non-zero finitely generated $A$-module. Let $\mathfrak{q}$ be an ideal of $A$ and $\au = a_1,\ldots,a_t$ a system of elements of $A$ such that $a_i \in \mathfrak{q}^{c_i}, c_i\in\mathbb{N},$ for
$i = 1,\ldots,t$. Also, assume that $n\in \mathbb{N}$. In \cite[Section 6]{KSch}, Schenzel and the author of the current manuscript have introduced a variation $\check{L}^i(\au,\mathfrak{q},M;n)$ of the local cohomology module $H^i_{\au A}(M)$, see Definition \ref{var3}. It turned out that $\check{L}^i(\au,\mathfrak{q},M;n)$ is also a generalization of $H^i_{\au A}(M)$: for the case $\q=\au A$, we have
\[
H^i_{\au A}(M) \cong \check{L}^i(\au,\au A,M;n) \text{ for all } n \gg 0 \text{ and } i \geq 0,
\]
see \cite[Remark 6.5]{KSch}. That is, in a certain sense $\check{L}^i(\au,\mathfrak{q},M;n)$ provides
some additional structure on the usual local cohomology module $H^i_{\au A}(M)$.

Furthermore, there is a well known criterion of Cohen-Macaulayness of $M$ in terms of (non-)vanishing of local cohomology modules. More precisely, $M$ is a Cohen-Macaulay $A$-module if and only if there is only one integer $i$ such that $H^i_{\mathfrak{m}}(M)\not=0$, see \cite[Corollary 6.2.9]{BS}.
The aim of this note is to present a similar criterion to check Cohen-Macaulayness of the form module $G_M(\q)=\oplus_{n \geq 0}\q^nM/\q^{n+1}M$ in terms of (non-)vanishing of $\check{L}^i(\au,\mathfrak{q},M;n)$. To this end assume that $\grade(\mathfrak{Q}, G_M(\q))$ denote the $G_M(\q)-grade$ of an ideal $\mathfrak{Q}$ of the form ring $G_A(\q)=\oplus_{n \geq 0}\q^n/\q^{n+1}$ which is defined as the length of a maximal homogeneous $G_M(\q)$--sequence in $\mathfrak{Q}$. Also assume that $a_i^\star=a_i+\q^{c_i+1}$ for $i=1,\ldots,t$ with $a_i\in \q^{c_i}\setminus \q^{c_i+1}$ and $\au^\star=a_1^\star,\ldots,a_t^\star$. We have the following main result:

\begin{theorem}\label{intro1}
	With the previous notations, $\grade(\au^\star G_A(\q), G_M(\q))$ is the least integer $i$ such that $\check{L}^i(\au,\q,M;n) \not= 0$ for some $n$.
\end{theorem}
We denote by $\mathfrak{M}=\mathfrak{m}/\q \oplus (\oplus_{n > 0}\q^n/\q^{n+1})$ the unique maximal homogeneous ideal of $G_A(\q)$. Then, $\grade(\mathfrak{M}, G_M(\q))$ is denoted by $\depth G_M(\q)$ and referred as the \emph{depth} of $G_M(\q)$. Note that if $a_1^\star,\ldots,a_t^\star$ is a system of parameters of $G_M(\q)$, then $\grade(\au^\star G_A(\q), G_M(\q))=\depth G_M(\q)$. 

\begin{corollary}\label{intro2}
	With the previous notations, if $a_1^\star,\ldots,a_t^\star$ is a system of parameters of $G_M(\q)$, then any integer $i\geq0$ for which $\check{L}^i(\au,\mathfrak{q},M;n)\not=0$, for some $n$, must satisfy
	\[
	\depth G_M(\q)\leq i\leq \dim G_M(\q).
	\]
	Moreover, $\check{L}^i(\au,\mathfrak{q},M;n)\not=0$, for some $n$, for $i=\depth G_M(\q)$ or $i= \dim G_M(\q)$.
\end{corollary}

\begin{corollary}\label{intro3}
With the previous notation, if $a_1^\star,\ldots,a_t^\star$ is a system of parameters of $G_M(\q)$, then there is only one integer $i$ for which $\check{L}^i(\au,\mathfrak{q},M;n)\not=0$, for some $n$, if and only if $G_M(\q)$ is a Cohen-Macaulay $G_A(\q)-$module.
\end{corollary}
The proofs of \ref{intro1}, \ref{intro2} and \ref{intro3} are given in \ref{c5}, \ref{c6} and \ref{c7} respectively. Note that for the proof of the above results, we follow the ideas from the book \cite{BS}.

This paper is organized as follows: Section 2 deals with basic notions from Commutative Algebra. The basic sources for this are \cite{AM,hM}. For results on Homological Algebra, we refer to \cite{cW}. We also discuss local cohomology in Section 2 which was developed in \cite{aG} (see also \cite{BS}). Section 3 is devoted to a variation of local cohomology introduced in \cite{KSch}. A further investigation around Lichtenbaum-Hartshorne Theorem in terms of this variation of local cohomology is in progress.

{\bf Acknowledgments: }The author is thankful to the reviewer for useful comments and suggestions. We are deeply grateful to Peter Schenzel for inspiring discussions during visit of the author at Max-Planck Institute of Mathematics in the Sciences, Leipzig (MPI MIS) in September 2018 and then later for useful comments and suggestions. We are thankful to Mateusz Micha\l ek and MPI MIS for the invitation and support for the visit to the institute. We are also thankful to Alex Dimca for discussion and helpful comments during the CIMPA School, Hanoi in March 2019. The author wants to express his gratitude to Tiberiu Dumitrescu and to his wife Aqsa Bashir for useful comments during the preparation of this manuscript.

\section{Preliminaries}
In this section, we fix notations which we use in the rest of the manuscript. Basics on $\mathbb{N}$-graded structures may be found in \cite{GW}.

\begin{notation} \label{not-1}
	(A) We denote by $A$ a commutative Noetherian ring with $0 \not= 1$ and $\mathfrak{q}$ an ideal of $A$. We write by $\au=a_1,\ldots,a_t$ a system of elements of $A$ such that $a_i\in\q^{c_i}, c_i\in\mathbb{N},$ for $i=1,\ldots,t$ and $M$ a non-zero finitely generated $A$-module.\\
	(B) We consider the Rees and form rings of $A$ with respect 
	to $\mathfrak{q}$ by 
	\[
	R_A(\mathfrak{q}) = \oplus_{n \geq 0} \mathfrak{q}^n \,T^n \subseteq A[T] \,
	\text{ and }\, G_A(\mathfrak{q}) = \oplus_{n \geq 0} \mathfrak{q}^n/\mathfrak{q}^{n+1}.
	\]
	Here $T$ denotes an indeterminate over $A$. Note that both rings are naturally $\mathbb{N}$-graded. We define 
	the Rees and form modules in the corresponding way by 
	\[
	R_M(\mathfrak{q}) = \oplus_{n \geq 0} \mathfrak{q}^n M \,T^n \subseteq M[T] \,
	\text{ and } \, G_M(\mathfrak{q}) = \oplus_{n \geq 0} \mathfrak{q}^nM/\mathfrak{q}^{n+1}M.
	\]
	Note that $R_M(\mathfrak{q})$ is a graded $R_A(\mathfrak{q})$-module and 
	$G_M(\mathfrak{q})$ a graded $G_A(\mathfrak{q})$-module. Also, $R_A(\mathfrak{q})$ 
	and $G_A(\mathfrak{q})$ are both Noetherian rings, and $R_M(\mathfrak{q})$ respectively $G_M(\mathfrak{q})$ is finitely generated 
	over $R_A(\mathfrak{q})$ respectively $G_A(\mathfrak{q})$.\\
	(C) Let $(A,\mathfrak{m})$ be a Noetherian local ring and $M$ a finitely generated $A$-module. It is well known that
	$\dim G_M(\q) = \dim M$.\\
	(D) There are the following two short exact sequences of graded modules
	\begin{gather*}
	0 \to R_M(\mathfrak{q})_{+}[1] \to R_M(\mathfrak{q}) \to G_M(\mathfrak{q}) \to 0 \text{ and }\\
	0 \to R_M(\mathfrak{q})_{+} \to R_M(\mathfrak{q}) \to M \to 0,
	\end{gather*}
	where $R_M(\mathfrak{q})_{+} = \oplus_{n > 0} \mathfrak{q}^n M \,T^n$. \\
	(E) If $m \in M$ such that $m \in \mathfrak{q}^c M \setminus \mathfrak{q}^{c+1} M$, then we define  
	$m^{\star} := m + \mathfrak{q}^{c+1} M \in [G_M(\mathfrak{q})]_c$ which is called the \emph{initial element} of $m$ in $G_M(\mathfrak{q})$ of the \emph{initial degree} $c$. If $m \in \cap_{n \geq 1} \mathfrak{q}^n M$, 
	then we write $m^{\star} = 0$.
	Here $[X]_n, n \in \mathbb{N},$ denotes the $n$-th graded component of an $\mathbb{N}$-graded module $X$.
\end{notation}
For these and related results, we refer to \cite{GW} and \cite{SH}. Another feature of our investigation is the use of Koszul and \v{C}ech complexes.

\begin{remark} \label{kos-1}
	(A) Let $\underline{a} = a_1,\ldots,a_t$ be a system of elements of a ring $A$. We define the \v{C}ech complex of $\au = a_1,\ldots,a_t$ by 
	\[
	\Cech : 0 \to \Cech^0 \to \ldots \to \Cech^i \to \ldots \to \Cech^t \to 0 \; \text{ with } \;
	\Cech^i = \oplus_{1 \leq j_1 < \ldots < j_i \leq t} A_{a_{j_1}\cdots a_{j_i}},
	\]
	where the differential $d^i : \Cech^i \to \Cech^{i+1}$ is given at the component 
	$A_{a_{j_1}\cdots a_{j_i}} \to A_{a_{j_1}\cdots a_{j_{i+1}}}$ by $(-1)^{k+1}$ times the 
	natural map $A_{a_{j_1}\cdots a_{j_i}} \to (A_{a_{j_1}\cdots a_{j_i}})_{a_{j_k}}$ if $\{j_1, \ldots,j_i\} 
	= \{j_1,\ldots,\widehat{j_k}, \ldots,j_{i+1}\}$ and zero otherwise. Moreover, we write 
	$\Cech(M) = \Cech \otimes_A M$. \\
	(B) The importance of the \v{C}ech complex is its relation to the local cohomology. Namely, let 
	$\mathfrak{q} = (a_1,\ldots,a_t)A$ be an ideal generated by a sequence $\au=a_1,\ldots,a_t$. The local 
	cohomology module $H^i_{\mathfrak{q}}(M)$ is defined as the $i$-th right derived functor 
	$H^i_{\mathfrak{q}}(M)$ of the section functor $\Gamma_{\mathfrak{q}}(M) = \{ m \in M\,|\,\Supp_A Am \subseteq V(\mathfrak{q})\}$. 
	Then there are natural isomorphisms 
	\[
	H^i_{\mathfrak{q}}(M) \cong H^i(\Cech \otimes_A M) \cong \varinjlim H^i(\au^n;M)
	\]
	for any $i \in \mathbb{Z}$, see \cite[Theorem 3.5.6]{BH} or \cite[Theorem 1.3]{pS}. As a consequence, it follows that $H^i(\Cech \otimes_A M)$ depends 
	only on the radical $\Rad \au A$.
\end{remark}	

\section{A variation of local cohomology}

In this section, we define a variation of local cohomology as introduced in \cite[Section 6]{KSch}. We also present a few properties of it which we need in the next section. To this end, assume that $M[\mathfrak{q}^c/a] \subseteq M_a$, when $a\in\q^c$ and $c\in \mathbb{N}$, consists of all elements of the form 
$m_0/1+q_1 m_1/a + \ldots + q_s m_s/a^s$ for some $s \geq 0$, $m_0,\ldots,m_s \in M$ and 
$q_i \in \mathfrak{q}^{ic}$ for $i = 1,\ldots,s$.


\begin{notation} \label{var1}(Khadam-Schenzel \cite[Notation 6.2 (A)]{KSch})
	Let $n\in\mathbb{N}$ be an integer. We define a complex $\check{C}^{\bullet}(\au,\mathfrak{q},M;n)$ in the following way:
	\begin{itemize}
		\item[(a)] For $0 \leq i \leq t$, put 
		$$
		\check{C}^i(\au,\mathfrak{q},M;n) = 
		\oplus_{1 \leq j_1 < \ldots < j_i \leq t} \mathfrak{q}^n M[\mathfrak{q}^{c_{j_1}+ \ldots + c_{j_i}}/(a_{j_1}\cdots a_{j_i})] \subseteq \oplus_{1 \leq j_1 < \ldots < j_i \leq t} M_{a_{j_1}\cdots a_{j_i}}
		$$
		and $\check{C}^i(\au,\mathfrak{q},M;n) = 0$ for $i < 0$ and $i > t$.
		\item[(b)] The boundary map $\check{C}^i(\au,\mathfrak{q},M;n) \to \check{C}^{i+1}(\au,\mathfrak{q},M;n)$ 
		is given by the restriction of the boundary map $\Cech(M) \to \Cech(M)$, see remark \ref{kos-1}(D).
	\end{itemize}
	Note that the restriction is a boundary map on $\check{C}^{\bullet}(\au,\mathfrak{q},M;n)$.
	We write $\check{H}^i(\au,\mathfrak{q},M;n), i\in \mathbb{Z},$ for the $i$-th cohomology of $
	\check{C}^{\bullet}(\au,\mathfrak{q},M;n)$.
	By the construction, it is clear that
	$\check{C}^{\bullet}(\au,\mathfrak{q},M;n)$ is a subcomplex of the \v{C}ech complex $\Cech(M)$.
\end{notation}

Another way of the construction is the following:

\begin{remark} \label{var2}
	Let $R_A(\mathfrak{q})$ and $R_M(\mathfrak{q})$ be the Rees ring and the Rees module respectively, and $\underline{aT^c} = a_1T^{c_1}, \ldots,a_tT^{c_t}$ a system of elements in $R_A(\mathfrak{q})$, where $a_i T^{c_i} \in [R_A(\mathfrak{q})]_{c_i}$ for $i = 1,\ldots,t$. Note that the degree $n$-component $\check{C}_{\underline{aT^c}}(R_M(\mathfrak{q}))_n$ 
	of the \v{C}ech complex $\check{C}_{\underline{aT^c}}(R_M(\mathfrak{q}))$ is the complex $\check{C}^{\bullet}(\au,\mathfrak{q},M;n)$ as introduced 
	in \ref{var1}. Hence $H^i_{\underline{aT^c}}(R_M(\mathfrak{q}))_n \cong \check{H}^i(\au,\mathfrak{q},M;n)$ for all $i$, see \cite[Section 6]{KSch} for details.
\end{remark}

\begin{definition}\label{var3}(\cite[Notation 6.2 (B)]{KSch})
	With the previous notations, we write $\check{\mathcal{L}}^{\bullet}(\au,\mathfrak{q},M;n)$ for the quotient of the embedding $\check{C}^{\bullet}(\au,\mathfrak{q},M;n) \to \Cech(M)$ of complexes. Whence,	there is a short exact sequence of complexes
	\[
	0 \to \check{C}^{\bullet}(\au,\mathfrak{q},M;n) \to \Cech(M) \to 
	\check{\mathcal{L}}^{\bullet}(\au,\mathfrak{q},M;n) \to 0.
	\]
	We write $\check{L}^i(\au,\mathfrak{q},M;n), i\in \mathbb{Z},$ for the $i$-th cohomology of 
	$\check{\mathcal{L}}^{\bullet}(\au,\mathfrak{q},M;n)$. Therefore, there is a long exact sequence of local cohomologies
	\[
	\ldots \to  \check{H}^i(\au,\mathfrak{q},M;n) \to H^i_{\au A}(M) \to \check{L}^i(\au,\mathfrak{q},M;n) \to \ldots.
	\]
\end{definition}



Next, we investigate the behaviour of these cohomologies under localization.

\begin{lemma}\label{var5}
	With the previous notations, if $S$ is a multiplicatively closed set in $A$, then
		\begin{gather*}
	S^{-1}\check{H}^i(\au,\mathfrak{q},M;n) \cong \check{H}^i(S^{-1}{\au},S^{-1}\mathfrak{q},S^{-1}M;n) \text{ and }\\
	S^{-1}\check{L}^i(\au,\mathfrak{q},M;n) \cong \check{L}^i(S^{-1}\au,S^{-1}\mathfrak{q},S^{-1}M;n)
	\end{gather*}
	for all $i$ and $n$, where $S^{-1}\au = a_1/1,\ldots,a_t/1$ is the system of elements in $S^{-1}A$.
\end{lemma}

\begin{proof}
	By considering $S^{-1} R_M(\q)$ as localization of an $A$-module, we get $S^{-1} R_M(\q) \cong R_{S^{-1}M}(S^{-1} \q)$. By flat base change of local cohomology, we get
	\[
	S^{-1} H^i_{\underline{aT^c}}(R_M(\mathfrak{q})) \cong H^i_{S^{-1} \underline{aT^c}}(R_{S^{-1}M}(S^{-1} \q)) \text{ and } 
	S^{-1} H^i_{\au A}(M) \cong H^i_{S^{-1}\au A}(S^{-1}M)
	\] 
	for all $i \in \mathbb{Z}$.
	Since
	\begin{gather*}
	[S^{-1} H^i_{\underline{aT^c}}(R_M(\mathfrak{q}))]_n \cong S^{-1}\check{H}^i(\au,\mathfrak{q},M;n) \text{ and }\\
	[H^i_{S^{-1} \underline{aT^c}}(R_{S^{-1}M}(S^{-1} \q))]_n \cong \check{H}^i(S^{-1}{\au},S^{-1}\mathfrak{q},S^{-1}M;n),
	\end{gather*}
	this proves the first isomorphisms. The second one follows by using the long exact sequence of cohomologies at the end of \ref{var3}.
\end{proof}

\begin{corollary}\label{var6}
	With the previous notations, the following holds:
	\begin{itemize}
		\item[(i)] $\Supp (\check{H}^i(\au,\mathfrak{q},M;n)) \subseteq \Supp M \cap V(\underline{a}A)$ and 
		\item[(ii)] $\Supp (\check{L}^i(\au,\mathfrak{q},M;n)) \subseteq \Supp M \cap V(\underline{a}A)$.
	\end{itemize}
\end{corollary}

\begin{proof}
	For the proof of (i), if $\mathfrak{p} \in \Supp (\check{H}^i(\au,\mathfrak{q},M;n))$, then by definition $\check{H}^i(\au,\mathfrak{q},M;n)_{\mathfrak{p}} \not = 0$. Hence $M_{\mathfrak{p}} \not = 0$ by previous lemma \ref{var5}. Also, it is easily seen that $\Supp (\check{H}^i(\au,\mathfrak{q},M;n)) \subseteq V(\underline{a}A)$, since otherwise $\check{C}_{\underline{aT^c}}(R_M(\mathfrak{q}))_n \otimes_A A_{\mathfrak{p}}$ is exact which is not possible.
	
	(ii) follows by the same line of reasoning.	
\end{proof}

\section{Criterion}
In this section, we present a criterion of Cohen-Macaulayness of the form module in terms of (non-)vanishing of $\check{L}^i(\au,\mathfrak{q},M;n)$. To prove the main theorem, we need a few technical results. We begin with the following lemma.

\begin{lemma}\label{c1}
	If $a \in \q^c$ and $b \in A$, then
	\[
	\q^n \overline{M}[\q^c / a] \cong (\q^n M[\q^c / a] + bM_a) / bM_a,
	\]
	where $\overline{M}=M/bM$.
\end{lemma}
\begin{proof}
	We have the following diagram
	\[
	\begin{array}{ccc}
	\q^n \overline{M}[\q^c / a] & {\longrightarrow} & \overline{M}_a \\
	&     & \downarrow {\scriptstyle{\cong}} \\
	(\q^n M[\q^c / a] + bM_a) / bM_a & {\longrightarrow} & M_a / bM_a,
	\end{array}
	\] 
	where the horizontal arrows are inclusion maps and the isomorphism is given by $(m + bM) / a^k \mapsto m / a^k + bM_a$.
	If $x \in \q^n \overline{M}[\q^c / a],$ then $x = q_0(m_0 + bM)/1 + q_1(m_1 + bM)/a + \ldots + q_s(m_s + bM)/a^s$ with $m_i \in M$ and $q_i \in \mathfrak{q}^{n+ic}$ for $i = 0,\ldots,s$. By applying above isomorphism, we get
	\begin{align*}
	x &= (q_0m_0/1 + bM_a) + (q_1m_1 / a + bM_a) + \ldots + (q_sm_s / a^s + bM_a) \\
	& = (q_0m_0/1 + q_1m_1 / a + \ldots + q_sm_s / a^s) + bM_a \in (\q^n M[\q^c / a] + bM_a) / bM_a.
	\end{align*}
	The reverse containment follows the same line of reasoning.
\end{proof}

There is another lemma.

\begin{lemma}\label{c2}
	If $a \in \q^c, b \in \q^d$ and $b^\star = b + \q^{d+1}$ is a $G_M(\q)-$regular element, then
	\[
	\q^{n+d} M[\q^c / a] :_{M_a} b = \q^n M[\q^c / a] \text{ for all } n.
	\]
\end{lemma}

\begin{proof}
Note that $b^\star = b + \q^{d+1}$ is a $G_M(\q)-$regular element if and only if $\q^{n+d}M :_M b = \q^nM$ for all $n$. Now, if $x / a^s \in \q^{n+d} M[\q^c / a] :_{M_a} b,$ then $bx / a^s \in \q^{n+d} M[\q^c / a]$ which can be written as $bx / a^s = q_0m_0/1 + q_1m_1 / a + \ldots + q_tm_t / a^t$ with $m_i \in M$ and $q_i \in \mathfrak{q}^{n+d+ic}$ for $i = 0,\ldots,t$. Hence there exists some non-negative integer $u$ such that
\[
a^{t+u}bx = a^{s+u}(q_0a^tm_0 + q_1a^{t-1}m_1 + \ldots + q_tm_t),
\]
where the right-hand side belongs to $(\mathfrak{q}^{n+d+(s+u)c}) (\mathfrak{q}^c,a)^t M  = \q^{n+d+(s+t+u)c}M$. That is, $a^{t+u}x \in \q^{n+d+(s+t+u)c}M :_M b$. Therefore, by assumption, $a^{t+u}x \in \q^{n+(s+t+u)c}M.$ Finally,
\[
x / a^s = a^{t+u}x / a^{s+t+u} \in \q^n ((\q^c / a)^{s+t+u}M) \subset \q^n M[\q^c / a].
\]
The reverse inclusion is obvious.
\end{proof}

As a consequence of previous two results, we present a result about exact sequence of complexes. For this, note that $m_b: \check{\mathcal{L}}^{\bullet}(\au,\q,M;n) \to \check{\mathcal{L}}^{\bullet}(\au,\q,M;n+d)$ defined as multiplication by $b$ is a morphism of complexes.


\begin{corollary}\label{c3}
		\begin{itemize}
		\item[(1)] With the previous notations, if $b \in \q^d$, then
		the sequence of complexes
		\[
		\check{\mathcal{L}}^{\bullet}(\au,\q,M;n) \stackrel{m_b} \to \check{\mathcal{L}}^{\bullet}(\au,\q,M;n+d) \to \check{\mathcal{L}}^{\bullet}(\au,\q,M/bM;n+d) \to 0
		\]
		is exact, and
		\item[(2)] if additionally $b^\star = b + \q^{d+1}$ is a $G_M(\q)$-regular element, then $m_b$ is an injective morphism of complexes.
		In particular, there is the long exact cohomology sequence
		\[
		\ldots \to \check{L}^i(\au,\q,M;n) \stackrel{b} \to \check{L}^i(\au,\q,M;n+d) \to \check{L}^i(\au,\q,M/bM;n+d) \to \ldots.
		\]
	\end{itemize}
\end{corollary}

\begin{proof}
	For statement (1), first note that
	\[
	\Coker(m_b) \cong \oplus_{1 \leq j_1 < \ldots < j_i \leq t} M_{a_{j_1}\cdots a_{j_i}} / (\mathfrak{q}^{n+d} M[\mathfrak{q}^{c_{j_1}+ \ldots + c_{j_i}}/(a_{j_1}\cdots a_{j_i})] + bM_{a_{j_1}\cdots a_{j_i}}).
	\]
	Now by using lemma \ref{c1}, we get
	\begin{align*}
	\Coker(m_b) &\cong \oplus_{1 \leq j_1 < \ldots < j_i \leq t} (M / bM)_{a_{j_1}\cdots a_{j_i}} / \mathfrak{q}^{n+d} (M / bM [\mathfrak{q}^{c_{j_1}+ \ldots + c_{j_i}}/(a_{j_1}\cdots a_{j_i})]) \\
	& \cong \check{\mathcal{L}}^i(\au,\q,M/bM;n+d)
	\end{align*}
	which completes the argument. 
	
	For statement (2), note that
	\[
	\Ker(m_b) \cong \oplus_{1 \leq j_1 < \ldots < j_i \leq t} \mathfrak{q}^{n+d} M[\mathfrak{q}^{c_{j_1}+ \ldots + c_{j_i}}/(a_{j_1}\cdots a_{j_i})] :_{M_{a_{j_1}\cdots a_{j_i}}} b / \mathfrak{q}^{n} M[\mathfrak{q}^{c_{j_1}+ \ldots + c_{j_i}}/(a_{j_1}\cdots a_{j_i})],
	\]
	which is equal to zero by using assumption and lemma \ref{c2}. Therefore, $\Ker(m_b)$ is a zero complex. This finishes the proof of (2).
\end{proof}

\begin{remark}\label{r1}
	Let $d_i:M\to M_{a_i}$ be the canonical map. For simplicity, we denote $d_i^{-1}(\q^nM[\q^{c_i}/a_i])\cap M$ by $\q^nM[\q^{c_i}/a_i]\cap M$. Then we have $\check{L}^0(\au,\q,M;n)\cong (\cap_{i=1}^t \q^nM[\q^{c_i}/a_i]\cap M) / \q^nM$. Note that $\q^nM[\q^{c_i}/a_i] \cap M = \cup_{k \in \mathbb{N}}\q^{n+kc_i}M : a_i^k = \q^{n+c_il_i}M :_M a_i^{l_i}$ for $l_i\gg0$ and $i=1,\ldots,t$, and hence $\check{L}^0(\au,\q,M;n)\cong \cap_{i=1}^t \q^{n+lc_i}M :_M a_i^l / \q^nM$, where $l:=\max\{l_1,\ldots,l_t\}$.
\end{remark}
Next we have a lemma concerning vanishing of the 0-th cohomology $\check{L}^0(\au,\q,M;n)$. As above, assume that $a_i^\star=a+\q^{c_i+1}$ for $i=1,\ldots,t$ with $a_i\in\q^{c_i}\setminus \q^{c_i+1}$ and $\au^\star=a_1^\star,\ldots,a_t^\star$.

\begin{lemma}\label{c4}
	With the previous notations, $\au^\star G_A(\q)$ contains a $G_M(\q)-$regular element if and only if $\check{L}^0(\au,\q,M;n) = 0$ for all $n \in \mathbb{N}$.
\end{lemma}

\begin{proof}Note that $\q^nM \subseteq \cap_{i=1}^t \q^{n+c_i}M :_M a_i \subseteq \cap_{i=1}^t \q^nM[\q^{c_i}/a_i]\cap M$. If $\check{L}^0(\au,\q,M;n) = (\cap_{i=1}^t \q^nM[\q^{c_i}/a_i]\cap M)/\q^nM = 0$ for all $n \in \mathbb{N}$, then by using above observation $\q^nM=\cap_{i=1}^t \q^{n+c_i}M :_M a_i$ for all $n$ and hence $[0 :_{G_M(\q)} \au^\star]_{n-1} = (\cap_{i=1}^t \q^{n+c_i}M :_M a_i \cap \q^{n-1}M) / \q^nM = 0$ for all $n$. Therefore $\au^\star G_A(\q)$ contains a $G_M(\q)-$regular element.
	
Conversely, if $\au^\star G_A(\q)$ contains a $G_M(\q)-$regular element, then the ideal $((a_1^\star)^k,\ldots,(a_t^\star)^k) \subseteq G_A(\q)$ also contains a $G_M(\q)-$regular element for all $k$, since $V(a^\star G_A(\q)) = V((a_1^\star)^k,\ldots,(a_t^\star)^k)$. That is, $[0 :_{G_M(\q)} ((a_1^\star)^k,\ldots,(a_t^\star)^k)]_{n-1} = 0$ for all $n$ and $k$. This implies that $\q^nM=\cap_{i=1}^t \q^{n+kc_i}M :_M a_i^k\cap \q^{n-1}M$ and hence it is easy to check, by induction on $n$, that $\q^nM=\cap_{i=1}^t \q^{n+kc_i}M :_M a_i^k$ for all $n$ and $k$. Therefore, by using above remark \ref{r1}, $\check{L}^0(\au,\q,M;n) = 0$ for all $n \in \mathbb{N}$. This finishes the proof.
\end{proof}

Let $\grade(\mathfrak{Q}, G_M(\q))$ be the $G_M(\q)-grade$ of an ideal $\mathfrak{Q}\subseteq G_A(\q)$ which is the length of a maximal homogeneous $G_M(\q)$--sequence in $\mathfrak{Q}$. We present the main result of the section. 

\begin{theorem}\label{c5}
	With the previous notations, $\grade(\au^\star G_A(\q), G_M(\q))$ is the least integer $i$ such that $\check{L}^i(\au,\q,M;n) \not= 0$ for some $n$.
\end{theorem}

\begin{proof}
	We use induction on $g:=\grade(\au^\star G_A(\q), G_M(\q))$. The case $g=0$ follows by previous lemma \ref{c4}. Now assume $g>0$ and that the result has been proved for every finitely generated $A-$module $N$ with $\grade(\au^\star G_A(\q), G_N(\q))<g$. 
	
	Therefore, there exists $b^\star=b+\q^{d+1} \in \au^\star G_A(\q)$ which is $G_M(\q)-$regular element. Moreover, by using \cite[Lemma 4.2]{K}, we can find an element $b'\in \q^d\setminus\q^{d+1}$ such that $b'^\star=b^\star$ and $b'\in \au A$. Hence, without loss of generality, we may assume that $b\in\au A$. By virtue of \cite[Section 1 and Proposition 2.1]{VV}, $G_M(\q)/b^\star G_M(\q)\cong G_{M/bM}(\q)$ and hence $\grade(\au^\star G_A(\q), G_{M/bM}(\q))=g-1$. Therefore, by inductive hypothesis, $\check{L}^i(\au,\q,M/bM;n+d)=0$ for all $i<g-1$ and for all $n$, while $\check{L}^{g-1}(\au,\q,M/bM;n+d) \not= 0$ for some $n$.
	
	Now, by applying Cor.\ref{c3} and by above computations, we get an exact sequence $0\to \check{L}^{g-1}(\au,\q,M/bM;n+d) \hookrightarrow \check{L}^g(\au,\q,M;n+d)$, which shows that 
	$\check{L}^g(\au,\q,M;n)\not=0$ for some $n$. Also, again by Cor.\ref{c3},
	\[
	\check{L}^i(\au,\q,M;n) \stackrel{b} \longrightarrow \check{L}^i(\au,\q,M;n+d) \text{ for all } i<g
	\]
	is injective. If $m \in \check{L}^i(\au,\q,M;n)$, then there exists an integer $\alpha$ such that $b^\alpha m=0$. Indeed, note that $\Supp(\check{L}^i(\au,\q,M;n)) \subseteq V(\au)$ and $b \in \au A$. But by injectivity, we get $m=0$. This finishes the proof.

\end{proof}

Note that if $\underline{b}=b_1,\ldots,b_l$ is another system of elements of $A$ with $b_i\in\q^{d_i}, d_i\in\mathbb{N},$ for $i=1,\ldots,l$ such that $\Rad \au A=\Rad \underline{b}A$, then $\check{L}^i(\au,\q,M;n)\cong \check{L}^i(\underline{b},\q,M;n)$ for all $i$ and for all $n$, see \cite[Lemma 6.4]{KSch}. But $\check{L}^i(\au,\q,M;n)$ may not be isomorphic to $\check{L}^i(\au,\Rad \q,M;n)$.

\begin{example}
	Let $k$ be a field and $A=k[|t^4,t^5,t^{11}|]\subseteq k[|t|]$, where $t$ is an indeterminate over $k$. Then $A$ is a one-dimensional domain and hence a Cohen-Macaulay ring. Also, note that $A\cong k[|X,Y,Z|]/(X^4-YZ,Y^3-XZ,Z^2-X^3Y^2)$ and $x$ the residue class of $X$ is a parameter of $A$. Hence $\Rad xA=\mathfrak{m}$. Now, $\check{L}^0(x,xA,A;n)=0$ for all $n$, since $(xA)^{n+k}:x^k=(xA)^n$ for all $k$ and for all $n$. On the other hand, $\check{L}^0(x,\mathfrak{m},A;n)\not=0$, for some $n$, by using above theorem \ref{c5}, since $G_A(\mathfrak{m})\cong k[X,Y,Z]/(XZ,YZ,Y^4,Z^2)$ is not Cohen-Macaulay (see \cite[Section 6]{V} for details).
\end{example}

From now on, we assume that $(A,\mathfrak{m})$ is a Noetherian local ring. We denote by $\mathfrak{M}=\mathfrak{m}/\q \oplus (\oplus_{n > 0}\q^n/\q^{n+1})$ the unique maximal homogeneous ideal of $G_A(\q)$. Then, $\grade(\mathfrak{M}, G_M(\q))$ is denoted by $\depth G_M(\q)$ and referred as the \emph{depth} of $G_M(\q)$. Note that if $\au^\star=a_1^\star,\ldots,a_t^\star$ is a system of parameters of $G_M(\q)$, then $\grade(\au^\star G_A(\q), G_M(\q))=\depth G_M(\q)$. 

\begin{corollary}\label{c6}
	With the previous notations, if $a_1^\star,\ldots,a_t^\star$ is a system of parameters of $G_M(\q)$, then any integer $i\geq0$ for which $\check{L}^i(\au,\mathfrak{q},M;n)\not=0$, for some $n$, must satisfy
	\[
	\depth G_M(\q)\leq i\leq \dim G_M(\q).
	\]
	Moreover, $\check{L}^i(\au,\mathfrak{q},M;n)\not=0$, for some $n$, for $i=\depth G_M(\q)$ or $i= \dim G_M(\q)$.
\end{corollary}

\begin{proof}
Note that if $a_1^\star,\ldots,a_t^\star$ is a system of parameters of $G_M(\q)$, then $a_1,\ldots,a_t$ is a system of parameters of $M$. Therefore, $\check{L}^i(\au,\mathfrak{q},M;n)$, for some $n$, is non-zero at $i=t$ and zero afterwards, see \cite[Theorem 9.5]{KSch}. The rest follows from above Theorem \ref{c5}.
\end{proof}

As a consequence of the previous result, we have a following criterion of Cohen-Macaulayness of the form module.
\begin{corollary}\label{c7}
With the previous notation, if $a_1^\star,\ldots,a_t^\star$ is a system of parameters of $G_M(\q)$, then there is only one integer $i$ for which $\check{L}^i(\au,\mathfrak{q},M;n)\not=0$, for some $n$, if and only if $G_M(\q)$ is a Cohen-Macaulay $G_A(\q)-$module.
\end{corollary}

\end{document}